\documentclass[12pt]{article}
\usepackage{amsfonts,amsmath,amsthm,amssymb}

\usepackage{graphicx}
\usepackage{setspace}
\usepackage{thmtools}

\usepackage{dsfont}
\usepackage[margin=.7in]{geometry}

\usepackage{xcolor}
\usepackage{hyperref}
\hypersetup{
    colorlinks = true,
    linkcolor = {blue}, urlcolor={blue}
}

\RequirePackage[numbers]{natbib}

\usepackage [english]{babel}
\usepackage [autostyle, english = american]{csquotes}
\MakeOuterQuote{"}

\newcommand{\PP}{\mathbb{P}}
\newtheorem{thm}{Theorem}

\newtheorem{lemma}{Lemma}

\numberwithin{equation}{section}
\numberwithin{thm}{section}
\numberwithin{lemma}{section}
\numberwithin{cor}{section}
\numberwithin{prop}{section}

\usepackage{fancyhdr}
\begin{document}

{
  \title{\bf Prime zeta function statistics and \\ Riemann zero-difference repulsion}
  \author{Gordon Chavez\footnote{gordon.v.chavez@gmail.com} \hspace{.1cm} and \hspace{.1cm} Altan Allawala}

\date{}
  \maketitle
} 

\begin{abstract}
We present a derivation of the numerical phenomenon that differences between the Riemann zeta function's nontrivial zeros tend to avoid being equal to the imaginary parts of the zeros themselves, a property called statistical "repulsion" between the zeros and their differences. Our derivation relies on the statistical properties of the prime zeta function, whose singularity structure specifies the positions of the Riemann zeros. We show that the prime zeta function on the critical line is asymptotically normally distributed with a covariance function that is closely approximated by the logarithm of the Riemann zeta function's magnitude on the 1-line. This creates notable negative covariance at separations approximately equal to the imaginary parts of the Riemann zeros. This covariance function and the singularity structure of the prime zeta function combine to create a conditional statistical bias at the locations of the Riemann zeros that predicts the zero-difference repulsion effect. Our method readily generalizes to describe similar effects in the zeros of related L-functions.
\end{abstract}

AMS 2010 Subject Classification: 11M06, 11M26

\tableofcontents

\listoftheorems

\newpage

\section{Introduction}
Numerical evidence from multiple sources has shown a notable phenomenon in the famous nontrivial zeros of the Riemann zeta function. It appears that differences between the Riemann zeros tend to avoid being equal to the imaginary parts of the Riemann zeros themselves. This long-range "repulsion effect" has been mentioned in Snaith's (2010) review and it was recently discovered in histograms of Riemann zero differences by Perez Marco (2011). Clear numerical evidence for this repulsion effect is visible in Figure \ref{fig0}, where a histogram of differences for the first 100,000 Riemann zeros clearly shows troughs around the imaginary parts of the Riemann zeros themselves.

\begin{figure}
\includegraphics[width=\linewidth]{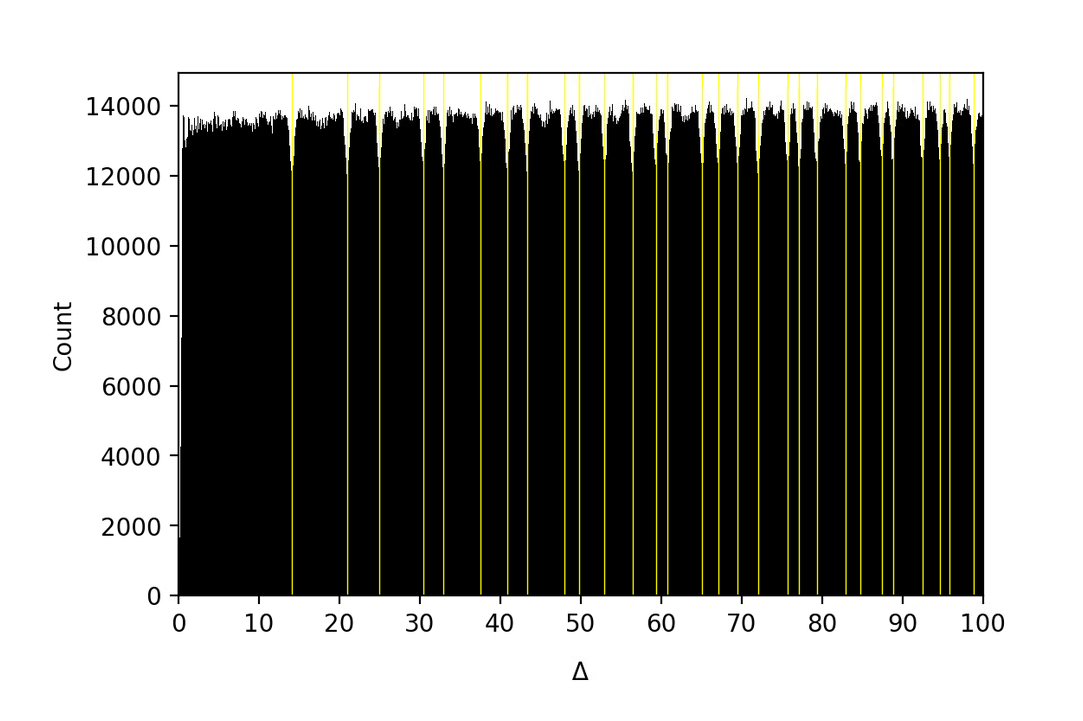}
\caption{A histogram of the differences for the first 100,000 Riemann zeros. Vertical lines are positioned at values of $\Delta$ such that $\zeta(1/2+i\Delta)=0$. Zeros were computed by Odlyzko (see e.g. 1988).}
\label{fig0}
\end{figure}

The study of the Riemann zeros' statistical properties is closely connected to the study of random matrices and quantum systems---in particular, systems with classical limits that exhibit chaotic dynamics. Hilbert and Polya first proposed that the Riemann zeros could be the eigenvalues of a linear, self-adjoint operator. This approach gained momentum when Dyson famously recognized that Montgomery's (1973) predictions for the Riemann zeros' correlations matched the correlations observed in the eigenvalues of random Hermitian matrices. Random matrices like these were likely familiar to Dyson from their use as Hamiltonians of large atomic nuclei. Odlyzko (1987) then showed that the empirical short range statistics of the Riemann zeros indeed matched those of such random matrix eigenvalues. 

The statistical properties of these random matrices follow the Gaussian Unitary Ensemble (GUE), which is known to describe quantum systems whose dynamics in the classical limit are chaotic and time-irreversible (Berry 1986). This connection to quantum chaos was strengthened by considering the Gutzwiller trace formula (1971), a semi-classical relation between a quantum system's energy eigenvalues and the corresponding classical system's periodic orbits. The famous Riemann-von Mangoldt explicit formula (1859), which relates the prime numbers and Riemann zeros, is notably analogous to Gutzwiller's trace formula, with the Riemann zeros playing the role of eigenvalues and the prime numbers playing the role of classical periodic orbits. Berry (1988) leveraged this analogy to provide more accurate predictions of long-range statistics in the Riemann zeros. The analogy is made still more striking and more precise by considering the Selberg zeta function (1956), a function describing dynamics generated by discrete groups on hyperbolic surfaces with constant curvature, whose structural similarity to the Riemann zeta function strongly suggests that the latter describes a quantum chaotic system on a constant negative curvature surface (Bogomolny 2007).

Bogomolny \& Keating (1996) utilized this analogy along with the Hardy-Littlewood conjecture from number theory to predict the zero-difference repulsion effect observed in Figure \ref{fig0}. Conrey \& Snaith (2008) then provided support for Bogomolny \& Keating's result by showing its derivation from a conjecture on the ratios of L-functions given by Conrey, Farmer, \& Zirnbauer (2008). Rodgers (2013) gave further support for Bogomolny \& Keating's result under the Riemann Hypothesis, and Ford and Zaharescu (2015) provided unconditional, number theoretic proof of the zero-difference repulsion effect.

In this paper we present an alternative, probabilistic derivation of this repulsion effect using the statistical properties of finite prime Dirichlet series. By a statistical independence property of the prime numbers, these series behave like sums of independent random variables, allowing us to invoke the central limit theorem for our purposes. We reach a related conclusion to previous research, showing that the repulsion effect originates from the influence of $\zeta(1+i\Delta)$. We show that the statistical properties of the \textit{prime zeta function} on the critical line, $P(1/2+i\tau)$, are described by a covariance function closely approximated by $\log \left|\zeta(1+i\Delta)\right|$. This produces negative covariance at separations $\Delta$ approximately equal to the imaginary parts of the Riemann zeros. By an important relation between $P(s)$ and $\log \zeta(s)$, the singularity structure of $P(1/2+i\tau)$ determines the location of Riemann zeros on the critical line. We then show that the combination of $P(1/2+i\tau)$'s singularity structure and covariance structure creates a conditional bias at the Riemann zeros, which predicts the zero-difference repulsion effect. We next describe how our methodology readily generalizes to predict the same effect in the complex zeros of related L-functions, and we close by noting our method's potential application to the Selberg zeta function and quantum chaotic systems.

\section{Covariance in Complex Prime Sums}
\label{sec1}

We first study the statistical properties of finite prime Dirichlet series with the form
\begin{equation}
X_{t}(\tau)=\sum_{p \leq t}a_{p}e^{i \left(\tau \log p+\theta_{p}\right)} \label{formula}
\end{equation} 
where the $a_{p}>0$ are real numbers indexed by the prime numbers $p$, $\tau \in \mathbb{R}$, and $\theta_{p} \in [0, 2\pi)$. It is clear that $EX_{t}(\tau)=0$ for $\tau$ uniformly distributed in $[T,2T]$ with $T\rightarrow \infty$. We will focus on (\ref{formula})'s covariance function, defined over separations $\Delta$, which is given by another such prime Dirichlet series.

\begin{thm}
Suppose $\tau$ is uniformly distributed in $[T,2T]$ with $T\rightarrow \infty$. Then (\ref{formula})'s summands are independent and for all $\Delta \in \mathbb{R}$
\begin{equation}
2R_{t}(\Delta)=\sum_{p \leq t}a_{p}^{2}\cos\left(\Delta \log p\right) \label{covariance}
\end{equation}
where
\begin{equation}
R_{t}(\Delta)=E\left\{\textnormal{Re}X_{t}(\tau+\Delta)\textnormal{Re}X_{t}(\tau)\right\}=E\left\{\textnormal{Im}X_{t}(\tau+\Delta)\textnormal{Im}X_{t}(\tau)\right\}. \label{covariancedefn}
\end{equation}
\label{thm0}
\end{thm}

\begin{proof}

We first derive the characteristic function of (\ref{formula})'s imaginary part. Recall that a random variable $x$'s characteristic function is defined 
\begin{equation}
\varphi_{x}(\lambda)=E\left\{e^{i\lambda x}\right\}. \label{chf}
\end{equation} 
We substitute (\ref{formula})'s imaginary part into (\ref{chf}) to write
\begin{equation}
\varphi_{\textnormal{Im}X_{t}}(\lambda)=E\left\{ \prod_{p \leq t}\exp\left( i a_{p}\lambda \sin\left(\tau \log p+\theta_{p}\right)\right) \right\}. \label{app1step00}
\end{equation}
We expand (\ref{app1step00}) using the Bessel function identity 
\begin{equation}
e^{ix \sin \phi}=\sum_{n=-\infty}^{\infty}J_{n}\left(x\right)e^{i n\phi} \label{besselidsin}
\end{equation}
where $J_{n}(.)$ is the $n$th-order Bessel function of the first kind (Laurin{\v c}ikas 1996). This gives
\begin{equation}
\varphi_{\textnormal{Im}X_{t}}\left(\lambda\right)=E\left\{\sum_{n_{1},...,n_{N}}J_{n_{1}}\left(a_{p_{1}}\lambda\right)...J_{n_{N}}\left(a_{p_{N}}\lambda\right)e^{i\left(n_{1}\theta_{p_{1}}+...+n_{N}\theta_{p_{N}}\right)}e^{i \tau\left(n_{1}\log p_{1}+...+n_{N}\log p_{N}\right)}\right\}. \label{big}
\end{equation}
The exponential terms on (\ref{big})'s far right-hand side are unit circle rotations with $\tau$. Therefore taking the expected value will cause all terms to vanish except those for which 
\begin{equation}
n_{1}\log p_{1}+...+n_{N}\log p_{N}=0. \label{zero}
\end{equation}
However, by unique-prime-factorization, the $\log p$'s are linearly independent over the rational numbers. Therefore the only solution to (\ref{zero}) is given by 
\begin{equation}
n_{1}=n_{2}=...=n_{N}=0. \label{soln}
\end{equation}
This simplifies (\ref{big}) to give
\begin{equation}
\varphi_{\textnormal{Im}X_{t}}\left(\lambda\right)=\prod_{p \leq t} J_{0}\left(a_{p}\lambda\right). \label{result}
\end{equation}
We next note from (\ref{chf}) and (\ref{besselidsin}) that the characteristic function for a single summand in (\ref{formula})'s imaginary part, $a_{p}\sin(\tau \log p+\theta_{p})$, is given by 
\begin{equation}
\varphi_{p}\left(\lambda\right)=J_{0}\left(a_{p}\lambda\right). \label{littleresult}
\end{equation}
Therefore, by (\ref{result}) and (\ref{littleresult}), 
\begin{equation}
\varphi_{\textnormal{Im}X_{t}}\left(\lambda\right)=\prod_{p\leq t} \varphi_{p}\left(\lambda\right). \label{independence}
\end{equation}
This proves that the summands of (\ref{formula})'s imaginary part are statistically independent. Essentially equivalent reasoning using the identity
\begin{equation}
e^{ix \cos \phi}=\sum_{n=-\infty}^{\infty} i^{n}J_{n}\left(x\right)e^{i n\phi}
\label{besselidcos}
\end{equation}
gives equivalent results for (\ref{formula})'s real part.
We next compute the covariance function for (\ref{formula})'s imaginary part. We first write  
\begin{align}
\textnormal{Im}X_{t}(\tau+\Delta)\textnormal{Im}X_{t}(\tau)=\nonumber \\ \sum_{p \leq t}a_{p}^{2}\sin\left((\tau+\Delta)\log p+\theta_{p}\right)\sin \left(\tau \log p+\theta_{p}\right)+2\sum_{\substack{p \neq q \\ p,q \leq t}}a_{p}a_{q}\sin\left((\tau+\Delta)\log p+\theta_{p}\right)\sin \left(\tau \log q+\theta_{q}\right).
\label{covstep1}
\end{align}
By the independence of (\ref{formula})'s summands from (\ref{result})-(\ref{independence}), the expected value of (\ref{covstep1})'s last summation vanishes. We then note that
\begin{equation}
\sin\left((\tau+\Delta)\log p+\theta_{p}\right)\sin \left(\tau \log p+\theta_{p}\right)=\frac{1}{2}\left(\cos\left(\Delta \log p\right)-\cos\left(\left(2\tau+\Delta \right)\log p +2\theta_{p}\right)\right). \label{covstep2term}
\end{equation}
The expected value of (\ref{covstep2term})'s second term vanishes. Applying this reasoning to (\ref{covstep1})'s first summation gives the result (\ref{covariance}). Essentially equivalent reasoning with the identity
\begin{equation}
\cos\left((\tau+\Delta)\log p+\theta_{p}\right)\cos \left(\tau \log p+\theta_{p}\right)=\frac{1}{2}\left(\cos\left(\Delta \log p\right)+\cos\left(\left(2\tau+\Delta \right)\log p +2\theta_{p}\right)\right)
\end{equation}
gives the result (\ref{covariance}) for (\ref{formula})'s real part as well.
\end{proof}
\hspace{-.6cm}Theorem \ref{thm0} shows that, for each prime $p$, the $p^{i\tau}$ are independent random variables. This is a direct consequence of unique-prime-factorization, which is otherwise known as the fundamental theorem of arithmetic. This independence property enables the straightforward evaluation of (\ref{formula})'s covariance function for separations $\Delta$.

Next we will apply Theorem \ref{thm0}'s results to study the repulsion effect in the nontrivial zeros of the Riemann zeta function. This will ultimately give results describing differences $\Delta$ between Riemann zeros. We stress that $\Delta$ in this context will \textit{not} indicate differences between \textit{consecutive} Riemann zeros, but rather, any two zeros on the critical line.

\section{Application to the Riemann Zeros}
\label{secapp}
We use the notation $P_{t}(s)=\sum_{p \leq t}p^{-s}$ and we denote the case with infinite $t$ as simply $P(s)$, which is the prime zeta function. We consider the particular series 
\begin{equation}
P_{t}(1/2+i\tau)=\sum_{p\leq t}\frac{1}{p^{1/2+i\tau}}, \label{logzetao1}
\end{equation}
which is an important object for research on the Riemann zeta function's behavior on the critical line (see e.g. Fyodorov \& Keating 2014, Arguin \& Tai 2019). We will apply the previous section's results to study (\ref{logzetao1})'s statistical properties. We first use the results of Section \ref{sec1} to show that (\ref{logzetao1})'s real and imaginary parts are asymptotically normally distributed.
\begin{lemma}
For $\tau$ uniformly distributed in $[T,2T]$ with $T\rightarrow \infty$ and $t \rightarrow \infty$,
\begin{equation}
\frac{1}{\sqrt{\frac{1}{2}\log\log t}}\textnormal{Re}P_{t}(1/2+i\tau) \hspace{.2cm} \textnormal{and} \hspace{.3cm}\frac{1}{\sqrt{\frac{1}{2}\log\log t}}\textnormal{Im}P_{t}(1/2+i\tau) \xrightarrow[]{d} \mathcal{N}(0,1). \label{theyrenormal}
\end{equation}\label{propnormal}
\end{lemma}
\begin{proof}
We note from (\ref{covariance}) that the variance of (\ref{logzetao1})'s real and imaginary parts is given by
\begin{equation}
\textnormal{var}\left\{\textnormal{Re}X_{t}(\tau)\right\}=\textnormal{var}\left\{\textnormal{Im}X_{t}(\tau)\right\}=R_{t}(0)=\frac{1}{2}\sum_{p \leq t}a_{p}^{2}. \label{appstatvar1}
\end{equation}
Setting $a_{p}=1/p^{1/2}$ then shows that 
\begin{equation}
\textnormal{var}\left\{\textnormal{Re}P_{t}(1/2+i\tau)\right\}=\textnormal{var}\left\{\textnormal{Im}P_{t}(1/2+i\tau)\right\}=\frac{1}{2}P_{t}(1)=\frac{1}{2}\log\log t+O(1) \label{appstatvar2}
\end{equation} 
by Mertens' 2nd theorem. We next note that since $P_{t}(1/2+i\tau)$'s summands are independent by Theorem \ref{thm0}, we may apply Lyapunov's Central Limit Theorem (see e.g. Billingsley 1995) with $\delta=2$, using (\ref{appstatvar2}) to give
\begin{eqnarray}
\frac{1}{\left(\frac{1}{2}P_{t}(1)\right)^{2}}\sum_{p \leq t}E\left\{\left(\frac{\cos(\tau\log p)}{p^{1/2}}\right)^{4}\right\}=\frac{1}{\left(\frac{1}{2}P_{t}(1)\right)^{2}}\sum_{p \leq t}E\left\{\left(\frac{\sin(\tau\log p)}{p^{1/2}}\right)^{4}\right\}\nonumber \\=\frac{C}{\left(\frac{1}{2}P_{t}(1)\right)^{2}}\sum_{p \leq t}\frac{1}{p^{2}}=\frac{O(1)}{\frac{\log^{2}\log t}{4}+O(\log \log t)}=O\left(\frac{1}{\log^{2}\log t}\right), \label{lyapunov}
\end{eqnarray}
where $C=E\cos^{4}(\omega \tau)=E\sin^{4}(\omega \tau)=12/32$ for arbitrary $\omega \in \mathbb{R}$. This completes the proof of (\ref{theyrenormal}), since (\ref{lyapunov}) vanishes as $t \rightarrow \infty$.
\end{proof}

The result (\ref{theyrenormal}) is quite similar to the well-known central limit theorem proven by Selberg (1946). This can be explained for the real part by the important fact that, with the proper truncation $t$, $\textnormal{Re}P_{t}\left(1/2+i\tau\right)$ converges in mean square and hence converges in probability to $\textnormal{Re}P\left(1/2+i\tau\right)$. More precisely, it can be shown from results given by Radziwill \& Soundararajan (2017) that
\begin{equation}
\textnormal{Re}P_{t^{X(t)}}(1/2+i\tau) \xrightarrow[]{p} \textnormal{Re}P(1/2+i\tau)\label{thankgod}
\end{equation}
as $t \rightarrow \infty$, where $t\leq \tau \leq 2t$ and $X(t)=1/(\log\log\log t)^{2}$. We will make use of this result below.

Since (\ref{logzetao1}) is asymptotically normally distributed, its statistical dependence structure is entirely described by its covariance function as $t \rightarrow \infty$. By (\ref{covariance}), the covariance function for (\ref{logzetao1}) is given by
\begin{equation}
2R_{t}(\Delta)=\sum_{p \leq t}\frac{\cos(\Delta \log p)}{p}=\textnormal{Re}P_{t}\left(1+i\Delta\right). \label{logzetavar}
\end{equation}
Numerical evidence suggests that the behavior of the series (\ref{logzetavar}) is closely connected to the imaginary parts of the Riemann zeros. This connection is visually evident in Figure \ref{fig2}, where we have computed (\ref{logzetavar}) using the first 1 million primes. It is clear that (\ref{logzetavar}) has minima approximately positioned at values of $\Delta$ such that $\zeta(1/2+i\Delta)=0$. These minima correspond to notable \textit{anti-correlations} over such separations $\Delta$. It is also clear that (\ref{logzetavar}) very closely follows the function $\log \left|\zeta(1+i\Delta)\right|$.

\begin{figure}
  \includegraphics[width=\linewidth]{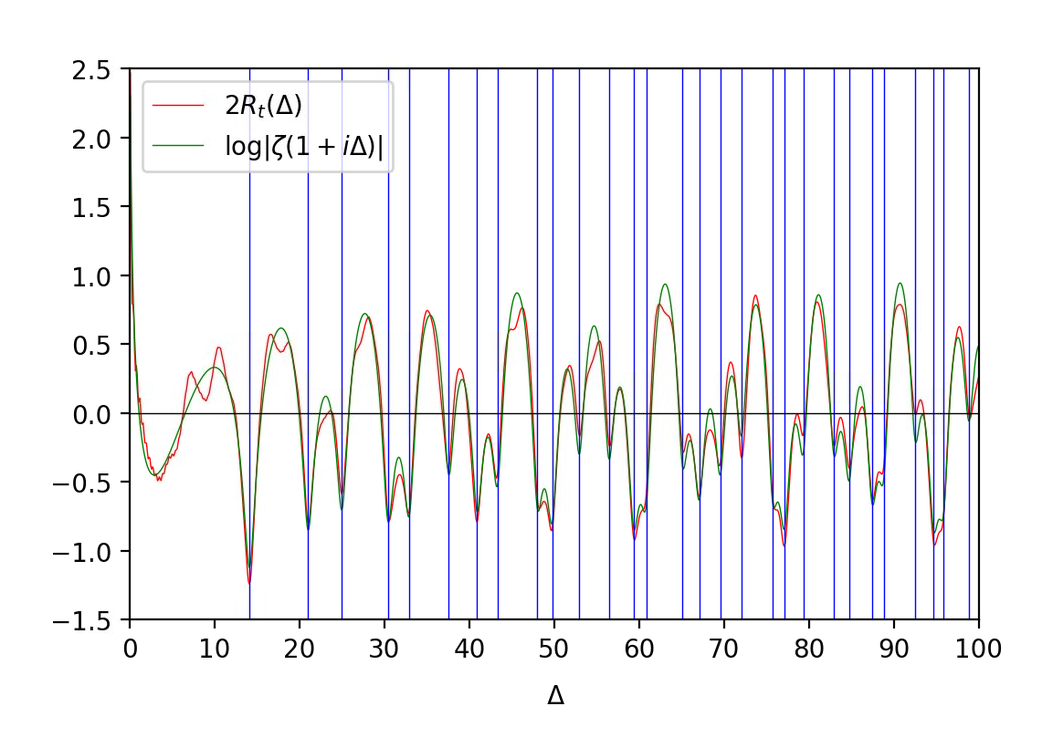}
  \caption{A graph of (\ref{logzetavar})'s $2 R_{t}(\Delta)$ with $t=p_{1,000,000}$ (Red) along with $\log \left|\zeta(1+i\Delta)\right|$ (Green) for $0\leq \Delta\leq 100$. Vertical lines are positioned at values of $\Delta$ such that $\zeta(1/2+i \Delta)=0$.}
  \label{fig2}
\end{figure}

To explain this phenomenon, we first show that (\ref{logzetavar}) converges in mean square and hence converges in probability to the full prime zeta function on the 1-line, $P(1+i\Delta)$, very quickly as $t$ increases.
\begin{lemma}
Suppose $\Delta$ is uniformly distributed in $[0,T]$ with $T\rightarrow \infty$. Then 
\begin{equation}
E\left\{\left(\textnormal{Re}P_{t}\left(1+i\Delta\right)-\textnormal{Re}P\left(1+i\Delta\right)\right)^{2}\right\}=\frac{P(2)-P_{t}(2)}{2} \rightarrow 0 \label{easyprop2}
\end{equation} \label{lemma1}
as $t \rightarrow \infty$.
\end{lemma}
\begin{proof}
Recall the independence of (\ref{formula})'s summands, that $E\cos\left(\omega \Delta\right)=0$, and $E\cos^{2}\left(\omega \Delta\right)=1/2$ for arbitrary $\omega$ to write
\begin{eqnarray}
E\left\{\left(\textnormal{Re}P(1+i\Delta)-\textnormal{Re}P_{t}(1+i\Delta)\right)^{2}\right\}=E\left\{\left(\sum_{p>t}\frac{\cos\left(\Delta \log p\right)}{p}\right)^{2}\right\}\nonumber \\
=E\left\{\sum_{p>t}\frac{\cos^{2}\left(\Delta \log p\right)}{p^{2}}\right\}+2E\left\{\sum_{\substack{p \neq q \\ p,q > t}}\frac{\cos\left(\Delta \log p\right)\cos\left(\Delta \log q\right)}{pq}\right\} \nonumber =\frac{1}{2}\sum_{p > t}\frac{1}{p^{2}},
\end{eqnarray}
which completes the proof of (\ref{easyprop2}).  
\end{proof}
\hspace{-.6cm}To demonstrate how quickly the mean squared error (\ref{easyprop2}) decreases, we note that its value for $t=100$ is less than .002 and for $t=1000$ is less than .0002. We next prove the following result for $P\left(1+i\Delta\right)$.
\begin{lemma}
For all $\Delta \in \mathbb{R}$, 
\begin{equation}
P\left(1+i\Delta\right)=\log \left|\zeta(1+i\Delta)\right|-\varepsilon\left(\Delta\right), \label{zetaconnectionreal}
\end{equation}
\label{zeta1cov}
where $\left|\varepsilon\left(\Delta\right)\right|<1-\gamma \approx .422784$. \label{lemma2}
\end{lemma} 
\begin{proof}
We first note the following result relating $P(s)$ and $\log \zeta(s)$ for $\textnormal{Re}(s)>0$, which can be derived by using the Euler product representation of $\zeta(s)$, taking the logarithm, and applying M{\"o}bius inversion (Glaisher 1891, Fr{\"o}burg 1968).
\begin{equation}
P(s)=\log\zeta(s)+\sum_{n=2}^{\infty}\frac{\mu(n)}{n}\log\zeta\left(n s\right), \label{primezetamobius}
\end{equation}
where $\mu(n)$ denotes the M{\"o}bius function. We next note from the Dirichlet series representation of $\zeta(n+in\Delta)$ that for $n\geq 2$, $\left|\zeta(n+in\Delta)-1\right|\leq \left|\zeta(n)-1\right|<1$. Therefore by the alternating series expansion of the logarithm,
\begin{equation}
\left|\log \zeta(n+in\Delta)\right|=\left|\sum_{k=1}^{\infty}\frac{(-1)^{k+1}}{k}\left(\zeta(n+in\Delta)-1\right)^{k}\right| < \left|\zeta(n+in\Delta)-1\right| \leq \zeta(n)-1. \label{prop21step1}
\end{equation}
Then by (\ref{prop21step1}) and since $\mu(n)\in \left\{-1,0,1\right\}$ for all $n$, we have
\begin{equation}
\left|\sum_{n=2}^{\infty}\frac{\mu(n)}{n}\log\zeta(n+in\Delta) \right|< \sum_{n=2}^{\infty}\frac{\left|\log\zeta(n+in\Delta) \right|}{n}<\sum_{n=2}^{\infty}\frac{\zeta(n)-1}{n}=1-\gamma. \label{prop21step2}
\end{equation}
(\ref{prop21step2})'s last equality can be shown by using the integral definition of $\zeta(n)$ for $n>1$ along with $\Gamma(n)$'s factorial and integral definition to write 
\begin{eqnarray}
\sum_{n=2}^{\infty}\frac{\zeta(n)-1}{n}=\sum_{n=2}^{\infty}\frac{1}{n!}\int_{0}^{\infty}x^{n-1}\left(\frac{1}{e^{x}-1}-\frac{1}{e^{x}}\right)dx=\int_{0}^{\infty}\frac{e^{x}-1-x}{xe^{x}\left(e^{x}-1\right)}dx \nonumber \\=\int_{0}^{\infty}\frac{dx}{e^{x}}-\int_{0}^{\infty}\left(\frac{1}{e^{x}-1}-\frac{1}{xe^{x}}\right)dx = 1-\gamma\label{sweetlawd} 
\end{eqnarray}
since (\ref{sweetlawd})'s last integral is an identity for $\gamma$ (Whittaker \& Watson 1990). Applying (\ref{prop21step2}) in (\ref{primezetamobius}) with $s=1+i\Delta$ completes the proof.
\end{proof}
\hspace{-.6cm}The results (\ref{easyprop2}) and (\ref{zetaconnectionreal}) show that, with large $t$, (\ref{logzetao1})'s covariance function (\ref{logzetavar}) is closely approximated by $\log\left| \zeta\left(1+i\Delta\right)\right|$, which creates negative covariance at separations $\Delta$ approximately equal to the imaginary parts of the Riemann zeros like that seen in Figure \ref{fig2}.

We next apply the formula (\ref{primezetamobius}) to the critical line to note that
\begin{equation} 
\textnormal{Re}P(1/2+i\tau)=\log \left|\zeta (1/2+i\tau)\right|+\sum_{n=2}^{\infty}\frac{\mu(n)}{n}\log \left|\zeta\left(n(1/2+i\tau)\right)\right|. \label{gotit}
\end{equation}
Recall that $\zeta(s)$ has no zeros with $\textnormal{Re}(s)\geq 1$ (Hadamard, de la Vall{\'e}e Poissin 1896) and its only pole is at $s=1$. Also, one can easily show that the terms of the series on (\ref{gotit})'s far right-hand side are $O\left(\frac{1}{n2^{n/2}}\right)$ as $n\rightarrow \infty$ and hence the series is convergent for any $\tau \neq 0$. Therefore, for $\tau \neq 0$, the singularities of (\ref{gotit}) unambiguously define the positions of $\zeta(1/2+i\tau)$'s zeros. With this context, we next consider the following limit theorem, which nearly completes our derivation of $\zeta(1/2+i\tau)$'s zero-difference repulsion effect.
\begin{thm}
Suppose $\tau$ is uniformly distributed in $[T,2T]$ with $T\rightarrow \infty$ and $\Delta$ is uniformly distributed in $\left(0,\tau \right]$. Then
\begin{equation}
p\left(\textnormal{Re}P\left(1/2+i\left(\tau+\Delta\right)\right)\middle \vert \zeta(1/2+i\tau)=0\right) \rightarrow \mathcal{N}\left(-\log \left|\zeta(1+i\Delta)\right|+\varepsilon(\Delta), \frac{1}{2}\log\log \tau\right) \label{fin!}
\end{equation}
where $p(.|.)$ denotes the conditional probability density function and $\left|\varepsilon\left(\Delta\right)\right|<1-\gamma \approx .422784.$ \label{mainresult}
\end{thm}
\begin{proof}
By Lemma \ref{propnormal}, both $\textnormal{Re}P_{t}(1/2+i\tau)$ and $\textnormal{Re}P_{t}\left(1/2+i\left(\tau+\Delta\right)\right)$ are normally and identically distributed as $t \rightarrow \infty$. We therefore use the following conditional mean formula for normally and identically distributed $x$ and $y$ with $E\left\{x\right\}=E\left\{y\right\}=0$:
\begin{equation}
E\left\{x|y\right\}=\textnormal{corr}\left\{x,y\right\}\times \frac{\textnormal{var}\left\{x\right\}}{\textnormal{var}\left\{y\right\}}\times y=\frac{\textnormal{cov}\left\{x,y\right\}\times y}{\textnormal{var}\left\{y\right\}} \label{normalgenform}
\end{equation}
Using (\ref{logzetavar}) to provide the covariance and variance terms in (\ref{normalgenform}) then gives
\begin{equation}
E\left\{\textnormal{Re}P_{t}(1/2+i(\tau+\Delta)))|\textnormal{Re}P_{t}(1/2+i\tau)\right\}=\frac{\textnormal{Re}P_{t}\left(1+i\Delta\right)\times\textnormal{Re}P_{t}(1/2+i\tau)}{P_{t}\left(1\right)}.  \label{condexp1}
\end{equation}
We next apply the bivariate normal conditional variance formula:
\begin{equation}
\textnormal{var}\left\{x|y\right\}=\textnormal{var}\left\{x\right\}\left[1-\left(\frac{\textnormal{cov}\left\{x,y\right\}}{\textnormal{var}\left\{x\right\}}\right)^{2}\right]=\textnormal{var}\left\{x\right\}-\frac{\left(\textnormal{cov}\left\{x,y\right\}\right)^{2}}{\textnormal{var}\left\{x\right\}}
\label{normalgenvar}
\end{equation}
Again using (\ref{logzetavar}) then gives
\begin{equation}
\textnormal{var}\left\{\textnormal{Re}P_{t}\left(1/2+i\left(\tau+\Delta\right)\right)|\textnormal{Re}P_{t}(1/2+i\tau)\right\}=\frac{1}{2}\left(P_{t}(1)-\frac{\left(\textnormal{Re}P_{t}\left(1+i\Delta\right)\right)^{2}}{P_{t}(1)}\right). \label{condvar}
\end{equation}
We next make the important note from (\ref{thankgod}) that (\ref{condexp1}) and (\ref{condvar}) are asymptotically equal to the conditional mean and variance for the full prime zeta function. More precisely, for $0<\Delta \leq \tau$ and $t=\tau^{1/(\log\log\log \tau)^{2}}$, as $\tau \rightarrow \infty$,
\begin{equation}
E\left\{\textnormal{Re}P_{t}(1/2+i(\tau+\Delta)))|\textnormal{Re}P_{t}(1/2+i\tau)\right\} \rightarrow E\left\{\textnormal{Re}P(1/2+i(\tau+\Delta)))|\textnormal{Re}P(1/2+i\tau)\right\} \label{thankgodmean}
\end{equation}
and
\begin{equation}
\textnormal{var}\left\{\textnormal{Re}P_{t}\left(1/2+i\left(\tau+\Delta\right)\right)|\textnormal{Re}P_{t}(1/2+i\tau)\right\}\rightarrow \textnormal{var}\left\{\textnormal{Re}P\left(1/2+i\left(\tau+\Delta\right)\right)|\textnormal{Re}P(1/2+i\tau)\right\}. \label{thankgodvariance}
\end{equation}
We then take the limit of (\ref{condexp1}) with $t=\tau^{1/\left(\log\log\log\tau\right)^{2}}$, $\tau \rightarrow \infty$, and with $\Delta$ uniformly distributed in $\left(0,\tau\right]$, using (\ref{thankgodmean}) as well as (\ref{easyprop2}) to write
\begin{equation}
E\left\{\textnormal{Re}P(1/2+i(\tau+\Delta)))|\textnormal{Re}P(1/2+i\tau)\right\} \xrightarrow[]{p}\frac{\textnormal{Re}P\left(1+i\Delta\right)\times\textnormal{Re}P(1/2+i\tau)}{\lim_{t\rightarrow \infty}P_{t}\left(1\right)}.  \label{condexp2}
\end{equation}
Note that taking $t \rightarrow \infty$ with $P_{t}(1)$ produces the singularity $P(1)=\log \zeta(1)+O(1)$ by (\ref{primezetamobius}) with $s=1$. Then noting the Laurent expansion of $\zeta(x)$ gives
$$\zeta(1)=\lim_{x\rightarrow 1^{+}}\frac{1}{x-1}=\lim_{x \rightarrow 0^{+}} \frac{1}{x}.$$
Applying this to (\ref{condexp2}) results in
\begin{eqnarray}
E\left\{\textnormal{Re}P(1/2+i(\tau+\Delta)))|\textnormal{Re}P(1/2+i\tau)\right\}\xrightarrow[]{p}\frac{\textnormal{Re}P\left(1+i\Delta\right)\times\textnormal{Re}P(1/2+i\tau)}{\lim_{x \rightarrow 0^{+}}\log \frac{1}{x}+O(1)}.
\label{condexp1.5}
\end{eqnarray}
For any finite value of $\textnormal{Re}P(1/2+i\tau)$, (\ref{condexp1.5})'s right-hand side will therefore vanish as $\tau \rightarrow \infty$. However, by (\ref{gotit}), at the Riemann zeros we have
\begin{eqnarray}
E\left\{\textnormal{Re}P(1/2+i(\tau+\Delta)))|\zeta(1/2+i\tau)=0\right\}\xrightarrow[]{p}\textnormal{Re}P\left(1+i\Delta\right)\frac{\lim_{x \rightarrow 0^{+}}\log x+\textnormal{l.t.}}{-\lim_{x \rightarrow 0^{+}}\log x+O(1)}\nonumber \\=-\textnormal{Re}P\left(1+i\Delta\right)=-\log \left|\zeta(1+i\Delta)\right|+\varepsilon(\Delta) \label{condexp2}
\end{eqnarray}
from Lemma \ref{zeta1cov}. We next take the limit of (\ref{condvar}) as $t \rightarrow \infty$. We first consider (\ref{condvar})'s leading term as well as the denominator of its second term, recalling Merten's 2nd theorem to write 
\begin{equation}
\textnormal{var}\left\{\textnormal{Re}P_{t}\left(1/2+i\left(\tau+\Delta\right)\right)|\textnormal{Re}P_{t}(1/2+i\tau)\right\} \rightarrow \frac{1}{2}\left(\log\log t-\frac{\left(P_{t}\left(1+i\Delta\right)\right)^{2}}{\log\log t+O(1)}\right)+O(1). \label{condvar2}
\end{equation}
We again use (\ref{easyprop2}) and note from (\ref{primezetamobius}) that, for $\Delta>0$, $P\left(1+i\Delta\right)$ has no singularities, only taking finite values. Therefore (\ref{condvar2})'s second term vanishes as $t \rightarrow \infty$. We hence take the limit of (\ref{condvar}) with $t=\tau^{1/\left(\log\log\log\tau\right)^{2}}$, $\tau \rightarrow \infty$, and with $\Delta$ uniformly distributed in $\left(0,\tau\right]$, using (\ref{thankgodvariance}) and (\ref{easyprop2}) to give
\begin{equation}
\textnormal{var}\left\{\textnormal{Re}P\left(1/2+i\left(\tau+\Delta\right)\right)|\textnormal{Re}P(1/2+i\tau)\right\} \xrightarrow[]{p}\frac{1}{2}\log\log \tau+o\left(\log\log \tau\right). \label{condvarfinal}
\end{equation}
From the conditional mean and variance results (\ref{condexp2}) and (\ref{condvarfinal}) as well as the asymptotic normality from Lemma \ref{propnormal} and (\ref{thankgod}) we therefore conclude (\ref{fin!}).
\end{proof}
(\ref{condexp1})-(\ref{condexp2}) shows that $\textnormal{Re}P(1/2+i(\tau+\Delta))$'s conditional expectation asymptotically vanishes at all $\tau \neq 0$ unless $\zeta(1/2+i\tau)=0$, where it then becomes closely approximated by $-\log \left|\zeta(1+i\Delta)\right|$. At such values of $\tau$, this results in $\textnormal{Re}P(1/2+i(\tau+\Delta))$'s conditional expectation having maxima at $\Delta$ approximately equal to the imaginary parts of Riemann zeros. The positions $\Delta$ of these maxima in the mean of (\ref{fin!}) imply a decreased probability that $\textnormal{Re} P\left(1/2+i(\tau+\Delta)\right)<L$ for any $L$. By (\ref{gotit}), this implies a decreased probability that $\left|\zeta\left(1/2+i(\tau+\Delta)\right)\right|=0$. The result (\ref{fin!}) thus predicts the zero-difference repulsion effect. 

A visualization is given in Figure \ref{fig3}, where (\ref{fin!})'s probability that $\textnormal{Re}P\left(1/2+i\left(\tau+\Delta\right)\right)$ takes an extreme negative value is graphed for $\tau$ given by the hundred thousandth Riemann zero ordinate. It is clear that the likelihood of such an extreme negative value, which is a necessary condition for $\zeta\left(1/2+i\left(\tau+\Delta\right)\right)$ to vanish by (\ref{gotit}), is much higher for regions of $\Delta$ away from the imaginary parts of the Riemann zeros. We lastly note that (\ref{fin!}) becomes the uniform distribution as $\tau \rightarrow \infty$. Hence this effect weakens higher up the critical line. 

\begin{figure}
  \includegraphics[width=\linewidth]{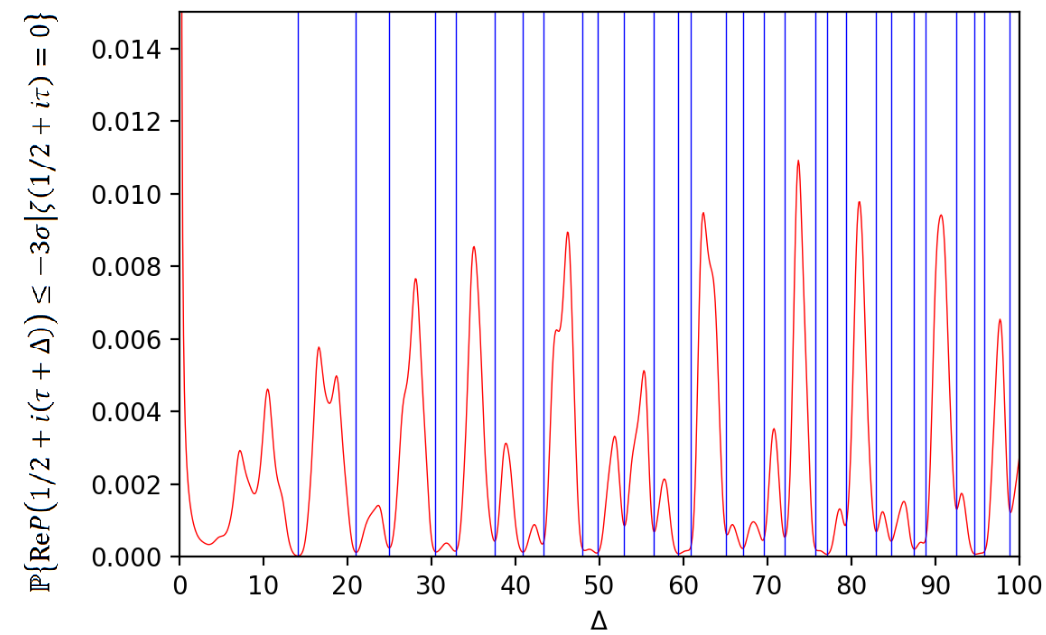}
  \caption{A graph of $\PP\left\{\textnormal{Re}P\left(1/2+i\left(\tau+\Delta\right)\right)\leq -3\sigma \middle \vert \zeta\left(1/2+i\tau\right)=0\right\}$, evaluated by integrating (\ref{fin!}), with $0<\Delta \leq 100$ and $\sigma=\frac{1}{2}\log\log \tau$, where $\tau$ is the imaginary part of the 100,000th Riemann zero.}
  \label{fig3}
\end{figure}

\section{Generalization to L-Function Zeros}
Perez Marco (2011) showed that a very similar effect occurs in many L-functions. In particular, he showed that the differences of the complex zeros of the L-functions $L\left(\chi_{3},s\right)$, $L\left(\chi_{4},s\right)$, and $L\left(\chi_{7,3},s\right)$ also seem to repel the Riemann zeros, indicating a connection between Riemann zeta and L-function zeros. We may explain this using our methodology by defining the following prime Dirichlet series
\begin{equation}
P_{\chi}(s)=\sum_{p}\frac{\chi(p)}{p^{s}} \label{primedirseries}
\end{equation}
for a given Dirichlet character $\chi$ s.t. $\chi(p)=e^{i\theta_{p}}$. With $s=1/2+i\tau$, Theorem \ref{thm0} and Lemma \ref{propnormal}'s reasoning may be applied to show that the real and imaginary parts of (\ref{primedirseries}) are normally distributed. We next note that for any Dirichlet character modulo $n$, $\chi_{n}$, where $n$ has prime factorization,
\begin{equation}
n=\prod_{p} p^{m_{p}}, \label{chibase}
\end{equation}
$\chi_{n}(p)=0$ for any $p$ such that $m_{p}\neq 0$, while for any $p$ with $m_{p}=0$, $\left|\chi_{n}(p)\right|=1$.
Therefore we may apply Theorem \ref{thm0} to show that, for $s=1/2+i\tau$, the covariance function for (\ref{primedirseries})'s real and imaginary parts  corresponding to the Dirichlet character $\chi_{n}$ have the same form as (\ref{logzetavar})-(\ref{zetaconnectionreal}) with  
\begin{equation}
\left|\varepsilon_{\chi_{n}}(\Delta)\right|<1-\gamma+\sum_{\substack{p\\m_{p}\neq0}}\frac{1}{p}. \label{chierror}
\end{equation}
For $\chi_{3}$, $\chi_{4}$, and $\chi_{7,3}$ these error bounds are relatively small since the only prime factors of these characters' moduli are 3, 2, and 7 respectively. Additionally, similar reasoning to that given for Lemma \ref{lemma1} shows that the mean squared difference between (\ref{primedirseries})-(\ref{chibase})'s covariance function and $P\left(1+i\Delta\right)$ is 
\begin{equation}
\frac{1}{2}\sum_{\substack{p\\m_{p}\neq 0}}\frac{1}{p^{2}}, \label{chimserror}
\end{equation}
which is quite small for $\chi_{3}$, $\chi_{4}$, and $\chi_{7,3}$ ($\approx$.056, .125, and .010 respectively). The error bound (\ref{chierror}) and the mean squared error (\ref{chimserror}) increase for moduli $n$ with a larger number of prime factors and are also larger for moduli with mostly small prime factors as opposed to moduli with only large prime factors. We lastly note that (\ref{primezetamobius}) may be generalized to give the following relationship between Dirichlet L-functions $L(\chi,s)$ and corresponding prime Dirichlet L-functions $P_{\chi}(s)$:
\begin{equation}
P_{\chi}(s)=\sum_{n=1}^{\infty}\frac{\mu(n)}{n}\log L\left(ns,\chi^{n}\right). \label{mobgen}
\end{equation}
Then a generalization of (\ref{thankgod}) for series of the form (\ref{primedirseries}), such as that provided from results of Hsu \& Wong (2019), enables derivation of an essentially equivalent result to (\ref{fin!}), which, by (\ref{mobgen}) and its corresponding generalization of (\ref{gotit}), has the same interpretation.

\section{Conclusion}
We have presented a derivation of zero-difference repulsion in the Riemann zeta function, which can also be applied to general L-functions. Our approach does not leverage the Riemann Hypothesis, Hardy-Littlewood conjecture, or other conjectures from number theory, and also does not directly use techniques from random matrix theory or quantum chaos, such as semi-classical trace formulae. Rather, our method relies on the statistical independence of the sequence $p^{i\tau}$ for all primes $p$, which is a direct consequence of the fundamental theorem of arithmetic. Thanks to this independence property in the prime numbers, we may invoke the central limit theorem and hence the properties of normally distributed random variables to study certain finite prime Dirichlet series. These series are asymptotically statistically equivalent to the full prime zeta function and hence our statistical results on the former may be applied to the latter. From there the only number theoretic concept we require is a classical analytic continuation technique using the Euler product and M{\"o}bius inversion, which relates the prime zeta function to the logarithm of the Riemann zeta function.

These methods readily generalize for L-functions and perhaps they can generalize further. In particular, the structural similarity between the Riemann zeta and Selberg zeta functions, including the connection between the former's prime numbers and the latter's prime geodesics or periodic orbits, as well as their shared Euler product structure, suggests that a generalization of the above derivations could be performed for the Selberg zeta function on certain hyperbolic surfaces. This could then help identify similar long-range repulsion effects in the eigenvalues of some quantum chaotic systems.

\end{document}